\documentclass[letterpaper, 10 pt, conference]{ieeeconf}  

\IEEEoverridecommandlockouts                              

\newtheorem{theorem}{Theorem}
\newtheorem{lemma}{Lemma}
\newtheorem{proposition}{Proposition}

\newtheorem{remark}{Remark}

\newtheorem{corollary}{Corollary}
\usepackage{hyperref}
\usepackage{amsmath} 
\usepackage{amssymb}  
\usepackage{arydshln}
\usepackage{mathtools}
\usepackage{customCommands}
\usepackage{tikz}
\usepackage{enumerate}
\usepackage{algorithm}
\usepackage{algpseudocode}
\usetikzlibrary{shapes,arrows, calc}
\newcounter{constnum}
\newcommand{\const}[1]{\refstepcounter{constnum}\label{#1}}

\usepackage{caption}
\usepackage{subcaption}

\def\bibtex{0}

\if\bibtex 1
\usepackage{cite}
\else
\usepackage[style=ieee, backend=bibtex]{biblatex}
\fi

\title{\LARGE \bf
Non-Asymptotic Analysis of Classical Spectrum Estimators with $L$-mixing Time-series Data
}

\begin{document}

\author{Yuping Zheng and
Andrew Lamperski
  \thanks{This work was supported in part by NSF ECCS 2412435}
  \thanks{Y. Zheng and A. Lamperski are with the department of Electrical and Computer Engineering, University  of Minnesota, Minneapolis, MN 55455, USA {\tt\small zhen0348@umn.edu, alampers@umn.edu}}
}

\maketitle
\thispagestyle{empty}
\pagestyle{empty}

\begin{abstract}
  Spectral estimation is a fundamental problem for time series analysis, which is widely applied in economics, speech analysis, seismology, and control systems. The asymptotic convergence theory for classical, non-parametric estimators, is well-understood, but the non-asymptotic theory is still rather limited. Our recent work gave the first non-asymptotic error bounds on the well-known Bartlett and Welch methods, but under restrictive assumptions. In this paper, we derive non-asymptotic error bounds for a class of non-parametric spectral estimators, which includes the classical Bartlett and Welch methods, under the assumption that the data is an $L$-mixing stochastic process.
 A broad range of processes arising in time-series analysis, such as autoregressive processes and measurements of geometrically ergodic Markov chains, can be shown to be $L$-mixing. In particular, $L$-mixing processes can model a variety of nonlinear phenomena which do not satisfy the assumptions of our prior work. Our new error bounds for $L$-mixing processes match the error bounds in the restrictive settings from prior work up to logarithmic factors. 
\end{abstract}

\section{INTRODUCTION}
%

Spectral estimation is the problem of estimating the power spectral density of a time series from a finite record of data samples. It is widely applied in economics, speech analysis, seismology, control systems, and more. The most common spectral estimation approaches are non-parametric (classical) methods and parametric (modern) methods \cite{stoica2005spectral}. Common non-parametric methods include periodograms, the Welch method, the Bartlett method, and the Blackman-Tukey method. Common parametric forms include ARMA and state space models.
Non-parametric estimators are widely used in practice, especially when information required to select a parametric model, such as autoregressive order, is unknown.
In this work, we focus on the analysis of two classical spectrum estimators: Bartlett and Welch estimators. 

The asymptotic analysis of spectral estimators, in which the length of the time series goes to infinity, is well-understood \cite{stoica2005spectral}, \cite{brillinger2001time}, \cite{brockwell1991time}, \cite{liu2010asymptotics}.
In practice, only a finite amount of data is available, and so non-asymptotic error bounds are desirable. For parametric estimation, non-asymptotic results have been long available for autoregressive models \cite{goldenshluger2001nonasymptotic}, while the non-asymptotic analysis of linear state-space models has become reasonably mature \cite{hardt2018gradient,oymak2018non,tsiamis2019finite,sarkar2019finite,lee2020non}.

The non-asymptotic theory of non-parametric spectral estimation is substantially less developed than either the asymptotic theory of non-parametric methods, or the non-asymptotic theory of parametric time series estimators. The most closely related work is \cite{lamperski2023nonasymptotic}, which gives non-asymptotic error bounds of the Bartlett, Welch, and Blackman-Tukey estimators, albeit under restrictive assumptions. Non-asymptotic bounds for estimators not covered in this paper include work on Blackman-Tukey methods \cite{zhang2021convergence,veedu2021topology}, smoothed periodograms \cite{fiecas2019spectral}, and Wiener filters \cite{doddi2022efficient}.



Our contribution is to obtain non-asymptotic error bounds for Bartlett and Welch estimators under the condition that the data series is $L$-mixing. The class of $L$-mixing processes was introduced in \cite{gerencser1989class} to quantify the decay of dependencies of stochastic processes over time. 
Many common models in time series analysis can be proved to be $L$-mixing, like autoregressive
processes \cite{gerencser1989class} and measurements of uniformly geometrically ergodic Markov chains \cite{gerencser2002new}.
In recent years, the theory of $L$-mixing processes has been used to give non-asymptotic error bounds on stochastic optimization methods with temporally dependent data streams
 \cite{chau2019stochastic}, \cite{zheng2022constrained}.
In Subsection~\ref{ss:comparison}, we give a detailed comparison between the $L$-mixing assumption and the assumptions from \cite{lamperski2023nonasymptotic}. 
In particular, we explain how $L$-mixing processes cover a variety of nonlinear phenomena (both in the dynamics and the measurements), which do not satisfy the assumptions in \cite{lamperski2023nonasymptotic}, 
while in many practical cases, data sequences satisfying the assumptions in \cite{lamperski2023nonasymptotic} are also $L$-mixing. 

The paper is organized as follows. In Section \ref{sec:Setup}, we set up the problem and present the algorithm. Section \ref{sec:analysis} presents the main results on spectral estimation error analysis. Section \ref{sec:L-mixing_detail} gives the background on $L$-mixing processes needed in the main proofs. Section \ref{sec:proofs} gives proofs of the main results. Section \ref{sec:sim} verifies our theory with a simulation of a finite-state Markov chain. Conclusions are given in Section~\ref{sec:conclusion}.
 
\section{Problem Setup} \label{sec:Setup}

\subsection{Notation}
The sets of real numbers, complex numbers and nonnegative integers are denoted by $\bbR$, $\bbC$ and $\bbN$ respectively.
Random variables are denoted in bold. If $\bx$ is a random variable, then 
$\bbE[\bx]$ is its expected value. $\|x\|_2$ denotes the Euclidean norm of the vector $x$. 
For matrix $A $, $\|A\|_2$ denotes the spectral norm and $\|A\|_F$ denotes the Frobenius norm. $\by^*$ denotes the conjugate transpose of $\by$. Let $j$ denote the imaginary unit. For real numbers $a$ and $b$, denote $a\lor b = \max\{a,b\}$.

Let $\cY$ be a finite-dimensional vector space with inner product $\langle \cdot,\cdot\rangle$ and corresponding norm $\|\cdot \|$. For a random variable, $\by\in\cY$, and $q\ge 1$ let $\|\by\|_{L_q}=\left(\bbE[\|\by\|^q]\right)^{1/q}$, which is the corresponding $L_q$ norm. 

\subsection{$L$-mixing processes} \label{subsec:L-mixing_intro}

We will assume that the data, $\by[k]$, is an $L$-mixing process. 
 Here, we introduce some background for $L$-mixing processes. We start with the classical definitions in continuous time and describe how they change for discrete time. 


 Let $\cF=(\cF_t)_{t\ge 0}$ be an increasing family of $\sigma$-algebras.
 Let $\cF^+=(\cF_t^+)_{t\ge 0}$ be a decreasing family of $\sigma$-algebras such that for all $t\ge 0$, $\cF_t$ and $\cF_t^+$ are independent,
$\cF_t^+=\cF_0^+$ for all $t\le 0$,
and $\cF_t^+=\sigma\left\{\bigcup_{\epsilon >0} \cF_{t+\epsilon} \right\}$. 
A continuous-time stochastic process $\by_t\in\cY$ is called \emph{$L$-mixing} with respect to $(\cF,\cF^+)$ if 
\begin{itemize}
\item $\by_t$ is measurable with respect to $\cF_t$ for all $t\ge 0$
\item $M_q(\by):=\sup_{t\ge 0}\|\by_t\|_{L_q}<\infty$ for all $q\ge 1$
\item $\Gamma_q(\by):=\int_0^{\infty}\gamma_q(\tau,\by)d\tau<\infty$ for all $q\ge 1$, where $\gamma_q(\tau,\by)=\sup_{t\ge \tau}\|\by_t-\bbE[\by_t|\cF_{t-\tau}^+]\|_{L_q}$. 
\end{itemize}
The number, $\Gamma_q(\by)$ characterizes the speed at which dependencies decay over time. 

Now we sketch the discrete-time case. 
Let $\cF=(\cF_k)_{k\ge 0}$ be an increasing sequence of $\sigma$-algebras and let $\cF^+=(\cF_k^+)_{k\ge 0}$ be a decreasing sequence of $\sigma$-algebras such that $\cF_k$ and $\cF_k^+$ are independent for all $k\ge 0$. 
We say that the discrete-time process $\by_k$ is $L$-mixing with respect to  $(\cF,\cF^+)$, if the continuous-time process defined by $\by_t=\by_{\floor{t}}$ is $L$-mixing with respect to the continuous-time $\sigma$-algebras defined by $\cF_t=\cF_{\floor{t}}$ and $\cF_t^+=\cF_{\floor{t}}^+$.

The moment bounds remains the same as in continuous time, but the discrete-time counterpart of $\Gamma_q(\by)$ is defined by
$
\Gamma_{d,q}(\by)=\sum_{k=0}^{\infty}\gamma_q(k,\by).
$	
Note that 
$\gamma_{q}(\tau,\by)=\max\{\gamma_q(\floor{\tau},\by),\gamma_q(\floor{\tau}+1,\by)\}$, which implies that
$
\Gamma_{d,q}(\by)\le \Gamma_q(\by)\le 2\Gamma_{d,q}(\by).
$
In particular, $\Gamma_{d,q}(\by)$ is finite  if and only if $\Gamma_q(\by)$ is finite.


%
%

\subsection{Problem and Algorithm}

Given a stationary zero-mean discrete-time stochastic process $\by[k] \in \bbR^n$, the autocovariance sequence and power sepctral density are 
$R[k] = \bbE[\by[i+k]\by[i]^\top]$ and $\Phi(s) = \sum_{-\infty }^{\infty} e^{-j 2 \pi s k} R[k]$,
where $s \in [-1/2, 1/2]$.

The batch-form 
Bartlett and Welch estimators are
\begin{align*}
  \hat\by_i(s)&=\sum_{k=0}^{M-1}w_k(s) \by[iK+k]\\
  \hat \bPhi_L(s)&=\frac{1}{L}\sum_{i=0}^{L-1}\hat \by_i(s)\hat \by_i(s)^\star
\end{align*}
where the Bartlett method has $K=M$ and $w_k(s)=\frac{1}{\sqrt{M}}e^{-j2\pi ks}$, while the Welch method may have $K\ne M$ and uses $w_k(s)=\frac{v_k}{\|v\|_2}e^{-j2\pi ks}$ for some window vector $0\ne v\in\bbR^M$.
In both the Bartlett and Welch estimators, $w(s)$ is a Euclidean unit vector for all $s$: $\sum_{k=0}^{M-1}|w_k(s)|^2=1$.

Choose $\alpha_k = \frac{1}{k+1}$, then the general iterative algorithm is:
$$
\hat \bPhi_{k+1}(s) = \hat \bPhi_k(s) + \alpha_k (\hat{\by}_{k}(s) \hat{\by}_{k}(s)^* - \hat \bPhi_k(s)).
$$

\subsection{Comparison with Assumptions from \cite{lamperski2023nonasymptotic}}
\label{ss:comparison}

In this paper, we assume that $\by[k]$ is an $L$-mixing sequence. In contrast, the work in \cite{lamperski2023nonasymptotic} requires that at least one of following two assumptions hold:
\begin{enumerate}[{A}1)]
\item
  \label{a:gaussian}
  $\by[k]$ is Gaussian
\item
  \label{a:subgaussian}
  There is an impulse  response sequence $h[k]\in\bbR^{n\times m}$ such that
  $\by[k] = \sum_{\ell=-\infty}^{\infty} h[k-\ell]\bzeta[\ell]$, where $\bzeta[k]=\begin{bmatrix}\bzeta_1[k] & \cdots & \bzeta_m[k]\end{bmatrix}^\top$ such that for $i=1,\ldots,m$ and for $k\in \bbZ$, $\bzeta_i[k]$ are independent $\sigma$-sub-Gaussian random variables.
\end{enumerate}


The class of $L$-mixing processes contains a wide variety of processes that cannot be modeled by the assumptions from \cite{lamperski2023nonasymptotic}, including measurements of geometrically ergodic Markov chains, \cite{gerencser2002new}, which can be used to model a wide variety of stable nonlinear stochastic systems. 
\footnote{The work in \cite{gerencser2002new} only proves the $L$-mixing property for measurements of uniformly geometrically ergodic Markov chains satisfying a $1$-step Doeblin condition. Based on our ongoing work, we hypothesize that measurements of irreducible, $V$-uniformly ergodic Markov chains are also $L$-mixing. The class of irreducible, $V$-uniformly ergodic Markov chains is broad, and covers nonlinear stochastic dynamical systems which satisfy a stochastic Lyapunov stability condition. See~\cite{meyn2012markov,douc2018markov}.} 
In particular, measurements of an ergodic Markov chain on a finite state space are $L$-mixing, but do not fit the assumptions from \cite{lamperski2023nonasymptotic}.  On the other hand, in many cases, the processes satisfying the assumptions from \cite{lamperski2023nonasymptotic} are also $L$-mixing. 
%
%
%

Furthermore, the class of $L$-mixing processes is closed under a variety of operations. In particular, 
if $\bzeta[k]$ is $L$-mixing, passing $\bzeta[k]$ through a stable, causal linear filter results in another $L$-mixing sequence. 
(Specific conditions on the filter are discussed in the result below.)
 If $f$ is Lipschitz, then $f(\bzeta[k])$ is $L$-mixing. The product of two $L$-mixing sequences is also $L$-mixing. As a result, rather complex processes can be shown to be $L$-mixing. 

The next result shows under suitable hypotheses on the filter $h$, data satisfying assumption A\ref{a:subgaussian} are also $L$-mixing. It is proved in Subsection~\ref{ss:pf:comparison}.

\begin{proposition}
  \label{prop:comparison}
  {\it
    If $\by[k]$ satisfies A\ref{a:subgaussian} and $h$ is causal with $\sum_{\ell=0}^{\infty}\|h[\ell]\|_2(\ell+1)<\infty$, then $\by[k]$ is $L$-mixing with
    \begin{align*}
      M_q(\by)&\le 8m\sigma \sqrt{q}\sum_{\ell=0}^{\infty} \|h[\ell]\|_{2}\\
\Gamma_{d,q}(\by)&\le 8m\sigma \sqrt{q}\sum_{\ell=0}^{\infty} \|h[\ell]\|_{2}(\ell+1).
    \end{align*}
}
\end{proposition}

\begin{remark}
In the case that $\by[k]$ is Gaussian, under the conditions that $\Phi(s)$ admits a spectral factorization (see \cite{wiener1957prediction}), $\by[k]$ also satisfies A\ref{a:subgaussian} with a causal filter $h$, $m=n$, and $\bzeta[k]$ Gaussian. In particular, when $\by[k]$ is generated by a stable linear Gaussian state space system, $h$ can be computed from the Kalman filter, and the hypotheses of Proposition~\ref{prop:comparison} hold. However, finding more general conditions to ensure that $h$ satisfies the requirements of Proposition~\ref{prop:comparison} is out of the scope of this paper.
\end{remark}
\section{Main Results of Convergence Analysis} \label{sec:analysis}

The main results show both the variance and bias bound of the estimators.
%
The first result bounds the averaged deviation between the algorithm and its mean value. 

\const{const_est}
\begin{theorem} \label{theorem_main}
  {\it
    Let $\by[k]$ be a zero-mean $L$-mixing sequence. Assume that $\alpha_j = \frac{1}{j+1}$, $\forall j \in \bbN $ and $j \in [0, k-1]$, 
    then for all integers $ k \ge 4$ and all $q\ge 1$:
\begin{align*}
\|\hat{\bPhi}_k(s) - \bbE[\hat{\bPhi}_k(s)]\|_{L_{q}} 
 &\le  b_q\frac{1}{\sqrt{k}} \log_2(\log_2 k)
\end{align*}
where $b_q  = 3 \cdot 2^{\frac{25}{4}}\sqrt{3(2q-1)} \left((4q-1)M_{4q}(\by)\Gamma_{d,4q}(\by)\right)^{\frac{3}{4}} \cdot \\
  \left( 4 \frac{M}{K}\sqrt{2(4q-1) M_{4q}(\by) \Gamma_{d,4q}(\by)} + 2 \Gamma_{d,4q}(\by) \right)^{\frac{1}{2}} \\
 + 32 (4q-1) M_{4q}(\by) \Gamma_{d,4q}(\by).$
%
}
\end{theorem}
%
%
The result of Theorem~\ref{theorem_main} holds in expectation. When the factor in the bound, $b_q$, grows polynomially with $q$, Theorem~\ref{theorem_main} implies a bound that holds in high probability. In the Markov chain example from Section~\ref{sec:sim}, and  also in the processes satisfying the hypotheses of Proposition~\ref{prop:comparison}, the polynomial growth assumption holds, and so the bounds can be computed explicitly.


\begin{theorem}
  \label{thm:highProb}
  {\it
    If there are constants $c>0$ and $r>0$ such that $b_q$ from Theorem~\ref{theorem_main} satisfies $b_q\le cq^r$ for all $q\ge 1$,
    then for all $\nu\in (0,1)$ and all $k\ge 4$:
    \begin{multline*}
      \bbP \left(
        \|\hat{\bPhi}_k(s)-\bbE[\hat{\bPhi}_k(s)]\|_F>\right.\\
      \left. c \frac{\log_2(\log_2(k))}{\sqrt{k}} e^r\max\left\{1,\frac{(\ln \nu^{-1})^{r}}{r^r}\right\}
      \right)
      \le \nu.
      \end{multline*}
  }
\end{theorem}

The following result is an immediate consequence of Proposition~\ref{prop:comparison} and Theorem~\ref{thm:highProb}. It enables direct comparison with the results from \cite{lamperski2023nonasymptotic} in the case when the data can also be shown to be $L$-mixing.

\begin{corollary}
  \label{cor:subGaussianMoments}
  {\it
If $\by[k]$ satisfies A\ref{a:subgaussian} and $h$ is causal with $\sum_{\ell=0}^{\infty}\|h[\ell]\|_2(\ell+1)<\infty$, then there are constants $c>0$ and $r>0$ such that $b_q\le cq^r$ for all $q\ge 1$. Consequently, the bound from Theorem~\ref{thm:highProb} holds. 
}
\end{corollary}

\begin{remark}
  The estimate, $\hat{\bPhi}_k(s)$ depends on the data $\by[0],\ldots,\by[(k-1)K+(M-1)]$. In other words, the total amount of data is $N=(k-1)K+(M-1)$. In the typical setup of Bartlett or Welch methods, $K\le M$, so that $k\le \frac{N}{K}$. It follows that, in terms of the total amount of data, the bound from Theorem~\ref{thm:highProb} is $O\left(\frac{\log_2(\log_2(N/K))}{\sqrt{N/K}} \right)$, which matches the corresponding bound from \cite{lamperski2023nonasymptotic} up to logarithmic factors.  
\end{remark}

 The spectral estimators are biased, meaning that $\Phi(s)\ne \bbE[\hat{\bPhi}_k(s)]$. So, full error bounds must also include bounds on the bias. In the case of $L$-mixing data, these bounds can be computed in terms of the mixing properties and parameters of the algorithms. 

\begin{proposition} \label{prop:bias}
  {\it
  If $\by[k]$ is $L$-mixing then:
  \begin{itemize}
  \item The bias of the Bartlett estimator is bounded by
  \begin{multline*}
    \|\Phi(s)-\bbE[\hat{\bPhi}_k(s)]\|_2\le  M_q(\by)\sum_{|k|\ge M}\gamma_2(|k|,\by)+ \\
    \frac{M_q(\by)}{M}\sum_{|k|<M} |k| \gamma_2(|k|,\by).
  \end{multline*}
  \item The bias of the Welch estimator is bounded by
  \begin{multline*}
    \|\Phi(s)-\bbE[\hat{\bPhi}_k(s)]\|_2\le  M_q(\by)\sum_{|k|\ge M}\gamma_2(|k|,\by)+\\
    M_q(\by)\sum_{|k|<M}\gamma_2(|k|,\by)\sum_{i=|k|}^{M-1}\frac{v_{i-|k|}v_i}{\|v\|_2^2}.
  \end{multline*}
  \end{itemize}
  }
\end{proposition}

The current bias bounds depend on the mixing properties in a rather complicated way. More explicit bounds, similar to those sketched in \cite{lamperski2023nonasymptotic} could be obtained, in particular, when $\gamma_2(k,\by)$ decreases exponentially. However, the corresponding formulas are  outside the scope of this work. 

\section{L-mixing Properties} \label{sec:L-mixing_detail}
This section is in preparation to prove Theorem~\ref{theorem_main}. 

\subsection{Variations on Classical $L$-mixing Results}

We present extensions of classical $L$-mixing results to the case of stochastic processes taking values in an arbitrary finite-dimensional inner  product space. We sketch the proofs following the methods in \cite{gerencser1989class}.

The following generalizes Lemma 2.3 of \cite{gerencser1989class}.
\begin{lemma}
  {\it
    \label{lem:innerProd}
  If $\by_t\in\cY$ is a zero-mean $L$-mixing process with respect to $(\cF,\cF^+)$ and $\bz\in\cY$ is $\cF_s$-measurable  with $s\le t$, then for any $p\ge 1$ and $q\ge 1$ with $\frac{1}{p}+\frac{1}{q}$, we have
  $$
  \left|\bbE\left[\langle \by_t,\bz\rangle\right] \right|\le 2\gamma_p(t-s,\by)\|\bz\|_{L_q}.
  $$
  }
\end{lemma}
\begin{proof}[Sketch]
  Follow the same steps as the proof  in \cite{gerencser1989class}, but replace multiplication with
 an inner product.
\end{proof}

The following result generalizes Theorem 1.1 of \cite{gerencser1989class} to vector-valued $\bu_t$ and complex-valued $f_t$.  
\begin{lemma}
  \label{lem:integral}
  {\it
    Let $\bu_t\in\cY$ be a zero-mean $L$-mixing process and let $f_t\in \bbC$. For all $m\ge 1$ and all $T\ge 0$:
    \begin{align*}
      \left\|\int_0^Tf_t \bu_t dt\right\|_{L_{2m}} \hspace{-18pt} \le 
      2\left((2m-1)M_{2m}(\bu)\Gamma_{2m}(\bu)\hspace{-3pt}\int_0^T \hspace{-3pt} |f_t|^2 dt\right)^{\frac{1}{2}}\hspace{-5pt}.
    \end{align*}    
  }
\end{lemma}
\begin{proof}[Sketch]
  Let $\bx_t = \int_0^t f_s \bu_s ds$. 
  The main difference from the proof in \cite{gerencser1989class} is that we get different expressions for the required derivatives of $\bx_t$:
\begin{align*}
  \frac{d}{dt}\|\bx_t\|^{2m}&=2m \|\bx_t\|^{2(m-1)} \Re\langle \bx_t,f_t \bu_t \rangle \\
  \frac{d}{dt}\left(\|\bx_t\|^{2(m-1)} \bx_t\right) &=\|\bx_t\|^{2(m-1)}\alpha_t \bu_t+\\
  &\hspace{-40pt}(m-1)\|\bx_t\|^{2(m-2)}\left(\langle \bx_t,f_t \bu_t\rangle + \langle f_t \bu_t ,\bx_t\rangle \right) \bx_t.
\end{align*}
After upper-bounding (via Lemma~\ref{lem:innerProd}), the same bound derived in Equation (2.9) of \cite{gerencser1989class} holds in this more general case. The proof is identical after that. 
\end{proof}

The next result follows from Lemma~\ref{lem:integral} applied to $\by_{t}=\by_{\floor{t}}$ and $w_t=w_{\floor{t}}$, and then using the bounds on $\Gamma_{d,q}(\by)$. 

\begin{corollary}
  {\it
    \label{cor:sum}
    Let $\by_k\in\cY$ be a zero-mean $L$-mixing discrete-time process and let $w_k\in\bbC$. For all $M\ge 1$ and all $q\ge 1$:
    \begin{align*}
    \left\|\sum_{k=0}^{M-1}w_k \by_k \right\|_{L_{2q}} \hspace{-15pt} \le 
      2 \left(2(2q-1)M_{2q}(\by)\Gamma_{d,2q}(\by)
        \sum_{k=0}^{M-1}|w_k|^2
     \right)^{\frac{1}{2}}\hspace{-5pt}.
    \end{align*}
  }
\end{corollary}

\subsection{$L$-mixing properties of Spectral Data Matrices}


%

The following result shows that the data matrices, $\hat \by_i(s)\hat \by_i(s)^\star$ inherit $L$-mixing properties from 
the original data sequence, $\by[k]$.

\begin{proposition}
  \label{prop:mixingMatrices}
  {\it
  If $\by[k]$ (using the Euclidean norm) is zero-mean and $L$-mixing with respect to $(\cF,\cF^+)$, then for all $s\in\bbR$, $\hat \by_i(s)\hat \by_i(s)^\star$ (using the Frobenius norm) is $L$-mixing with respect to $(\cG,\cG^+)$ where $\cG_i=\cF_{iK+(M-1)}$ and $\cG_i^+=\cF_{iK+(M-1)}^+$ for all $i\in\bbN$. Furthermore, the bounds satisfy
  \begin{align*}
    M_q(\hat\by(s)\hat\by(s)^\star)
    &\le
      8(2q-1)M_{2q}(\by)\Gamma_{d,2q}(\by)
    \\
    \Gamma_{d,q}(\hat\by(s)\hat\by(s)^\star)&\le 12\sqrt{2(2q-1)M_{2q}(\by)\Gamma_{d,2q}(\by)}\cdot\\
    &
      \hspace{-30pt}\left(4\frac{M}{K}\sqrt{2(2q-1)M_{2q}(\by)\Gamma_{d,2q}(\by)}+\Gamma_{2q}(\by)\right)
  \end{align*}
  for all $q\ge 1$.
  }
\end{proposition}

In the rest of this section, we will prove Proposition~\ref{prop:mixingMatrices}. And we will drop the dependence on $s$ for compact notation.

Firstly, we show that the vectors $\hat\by_i$ are also $L$-mixing.
%
 \begin{lemma}
  \label{lem:transformMixing}
  {\it
   If $\by[k]$ is zero-mean and $L$-mixing with respect to $(\cF,\cF^+)$, then $\hat \by_i$ is $L$-mixing with respect to $(\cG,\cG^+)$ where $\cG_i=\cF_{iK+(M-1)}$ and $\cG_i^+=\cF_{iK+(M-1)}^+$ for all $i\in\bbN$. Furthermore, for all $q\ge 1$, the bounds satisfy
  \begin{align*}
    M_q(\hat\by)&\le 2\sqrt{2((q\lor 2)-1)M_{q\lor 2}(\by)\Gamma_{d,q\lor 2}(\by)} \\
    \Gamma_{d,q}(\hat\by)&\le
                           2\frac{M}{K} M_q(\hat \by) + \Gamma_q(\by).
  \end{align*}
 } 
\end{lemma}

\begin{proof}
  For $q\ge 1$, we use Corollary~\ref{cor:sum} and the fact that $w$ is a unit vector to give:
  \begin{align*}
    \left\|\hat \by_i\right\|_{L_{2q}}&=\left\| \sum_{k=0}^{M-1}w_k\by[iK+k]\right\|_{L_{2q}}\\
    &\le 2\sqrt{2(2q-1)M_{2q}(\by)\Gamma_{d,2q}(\by)} 
  \end{align*}

  When $q\ge 2$, the bound on $M_q(\hat \by)$ follows by replacing $2q$ with $q$. For $q\in [1,2)$, the result follows from monotonicity of the $L_q$ norms. 

  Now we bound $\Gamma_{d,q}(\hat\by)$.
  By construction, $\hat\by_i$ is $\cG_i$ measurable for all $i$ and $\cG_i$ and $\cG_i^+$ are independent. For all $0\le \ell \le i$, we have
  $
  \|\hat\by_i-\bbE[\hat\by_i|\cG_{i-\ell}^+]\|_{L_q} \le 2 M_q(\hat \by).
  $
  When $\ell K \ge (M-1)$, we have that $(i-\ell)K+(M-1)=\left(iK+k\right)-\left(\ell K+k-(M-1)\right)$, where $\ell K+k-(M-1)\ge 0$ for all $k=0,\ldots,M-1$. In this case, using the triangle inequality, and that $|w_k|\le 1$ gives:
  \begin{align*}
    \MoveEqLeft[0]
    \|\hat\by_i-\bbE[\hat\by_i|\cG_{i-\ell}^+]\|_{L_q}\\
    &\le 
      \sum_{k=0}^{M-1}\|
      \by[iK+k]-\bbE[\by[iK+k]|\cF_{(i-\ell)K+(M-1)}^+]\|_{L_q} \\
    &\le \sum_{k=0}^{M-1}\gamma_{q}(\ell K +k-(M-1),\by).
  \end{align*}
  Now, since at most $M/K$ values of $\ell$ have $\ell < (M-1)/K$, the bound on $\Gamma_q(\hat \by)$ follows by summing over $\ell$.
\end{proof}

Now we can use the $L$-mixing properties of $\hat \by_i$ to derive the $L$-mixing properties of the outer product, $\hat\by_i\hat\by_i^\star$.

\paragraph*{Proof of Proposition~\ref{prop:mixingMatrices}}

Note that for any vectors, $\|xy^\star\|_F=\|x\|_2 \|y\|_2$. In particular,
$
\|\hat \by_i \hat \by_i^\star\|_{L_q}=\| \|\hat\by_i\|_2^2\|_{L_q}=\|\hat\by_i\|_{L_{2q}}^2. 
$
The bound on $M_{q}(\hat \by \hat \by^\star)$ now follows.

To bound the mixing constant, first note that for $0\le \ell\le i$:
\begin{align*}
  \MoveEqLeft
  \hat\by_i\hat\by_i^\star - \bbE\left[\hat\by_i\hat\by_i^\star\middle|\cG_{i-\ell}^+ \right]
  \\
  &=(\hat\by_i-\bbE[\hat\by_i|\cG_{i-\ell}^+])(
    \hat\by_i-\bbE[\hat\by_i|\cG_{i-\ell}^+])^\star\\
  &+(\hat\by_i-\bbE[\hat\by_i|\cG_{i-\ell}^+])\bbE[\hat\by_i|\cG_{i-\ell}^+]^\star\\
  &+
    \bbE[\by_i|\cG_{i-\ell}^+]
    (\hat\by_i-\bbE[\hat\by_i|\cG_{i-\ell}^+])^\star\\
  &-\bbE\left[(\hat\by_i-\bbE[\hat\by_i|\cG_{i-\ell}^+])(
    \hat\by_i-\bbE[\hat\by_i|\cG_{i-\ell}^+])^\star\middle|\cG_{i-\ell}^+\right].
\end{align*}
This equality is derived by plugging in $\hat\by_i=(\hat\by_i-\bbE[\hat\by_i|\cG_{i-\ell}^+])+\bbE[\hat\by_i|\cG_{i-\ell}^+]$ and then simplifying.

Now, taking the $L_q$ norm gives
\begin{multline}
  \label{eq:matrixTriangle}
  \left\|
  \hat\by_i\hat\by_i^\star - \bbE\left[\hat\by_i\hat\by_i^\star\middle|\cG_{i-\ell}^+ \right]\right\|_{L_q}
  \\
  \le 2\left\|(\hat\by_i-\bbE[\hat\by_i|\cG_{i-\ell}^+])\bbE[\hat\by_i|\cG_{i-\ell}^+]^\star\right\|_{L_q}\\
  +  2 \left\|(\hat\by_i-\bbE[\hat\by_i|\cG_{i-\ell}^+])( \hat\by_i-\bbE[\hat\by_i|\cG_{i-\ell}^+])^\star\right\|_{L_q}                                                                            .    
\end{multline}

The first term on the right can be bounded by:
\begin{align*}
  &\left\|\left\|\hat\by_i-\bbE[\hat\by_i|\cG_{i-\ell}^+]\right\| \left\|\bbE[\hat\by_i|\cG_{i-\ell}^+] \right\| \right\|_{L_q}\\
  &\le  \left\|\hat\by_i-\bbE[\hat\by_i|\cG_{i-\ell}^+]\right\|_{L_{2q}}
    \left\|\bbE[\hat\by_i|\cG_{i-\ell}^+] \right\|_{L_{2q}} \\
  &\le \gamma_{2q}(\ell,\hat\by)M_{2q}(\hat\by).
\end{align*}

The second term on the right is bounded similarly by:
\begin{align*}
   & \left\|\left\|\hat\by_i-\bbE[\hat\by_i|\cG_{i-\ell}^+]\right\| \left\|\hat\by_i-\bbE[\hat\by_i|\cG_{i-\ell}^+]\right\| \right\|_{L_q} \\
&\le 2\|\hat\by_i\|_{L_{2q}}  \|\hat\by_i-\bbE[\hat\by_i|\cG_{i-\ell}^+]\|_{L_{2q}} \\
  &\le 2M_{2q}(\hat\by)\gamma_{2q}(\ell,\hat\by). 
\end{align*}

Plugging the bounds into (\ref{eq:matrixTriangle}) and summing over $\ell$ gives:
$$
\Gamma_{d,q}(\hat\by\hat\by^\star)\le 6M_{2q}(\hat \by)\Gamma_{d,2q}(\hat \by).
$$
The result now follows by plugging in the expressions for $M_{2q}(\hat \by)$ and $\Gamma_{d,2q}(\hat\by)$ from Lemma~\ref{lem:transformMixing}. 
\hfill\QED

\section{Proofs of Main Results} \label{sec:proofs}
After obtaining the necessary L-mixing properties, we are ready to step through proofs of the main results.
\subsection{Proof of Theorem ~\ref{theorem_main}}
Let $\bx_{k} = \hat{\bPhi}_k(s)$, $\bz_k = \hat{\by}_{k}(s) \hat{\by}_{k}(s)^*$, $\bar{z} = \bbE[\bz_k]$, $\be_k =  \bx_k -\bar{z}$ and $\tilde{\bz}_k = \bz_k -\bar{z}$.
Then, $\be_{k+1} = (1-\alpha_k) \be_k  + \alpha_k \tilde{\bz}_k$.

Setting $\beta_{i,k} = \Pi_{j = i+1}^{k-1}(1-\alpha_j) \alpha_i \le \alpha_i$, $\tau_i = \sum_{j=0}^{i-1} \alpha_j$ and for all $q \ge 1 $, iterating and taking the $L_{2q}$ norm gives
\begin{align*}
\nonumber
\|\be_k \|_{L_{2q}} &\le \Pi_{i = k_0}^{k-1} (1-\alpha_i) \|\be_{k_0}\|_{L_{2q}}  + \| \sum_{i = k_0}^{k-1} \beta_{i,k} \tilde{\bz}_i  \|_{L_{2q}} \\ \nonumber
& \le e^{-(\tau_k - \tau_{k_0})} \|\be_{k_0}\|_{L_{2q}} \\ \nonumber
& \quad + \sqrt{ \sum_{i = k_0}^{k-1} \beta_{i,k}^2} 2 \sqrt{2(2q-1) M_{2q}(\tilde{\bz}) \Gamma_{d,2q}(\tilde{\bz})} \\ 
& \le e^{-(\tau_k - \tau_{k_0})} \|\be_{k_0}\|_{L_{2q}} + \sqrt{ \sum_{i = k_0}^{k-1} \alpha_i^2} C_q, 
\end{align*}
where $C_q = 2 \sqrt{2(2q-1) 2 M_{2q}(\bz) \Gamma_{d,2q}({\bz})} $.
%
%
The first inequality above uses $1- x \le e^{-x}$, $\forall x \in \bbR$ and Corollary~\ref{cor:sum}.

Choose integers $1=s_0<\cdots < s_m < s_{m+1}= k$ and iterate the bound above:
%
\begin{align*}
\|\be_k \|_{L_{2q}} &\le C_q \sum_{l=0}^{m} \sqrt{\sum_{i =s_{l}}^{s_{l+1}-1} \alpha_i^2} e^{-(\tau_k - \tau_{s_{l+1}})} \\
& \hspace{60pt} + e^{-(\tau_k-\tau_{s_0}) } \|\be
_{s_0}\|_{L_{2q}} \\
 &\le 2 C_q \sum_{l=0}^{m} \frac{1}{\sqrt{s_l}} \frac{s_{l+1}}{k} +    \frac{2}{k} \|\be_1\|_{L_{2q}}.
\end{align*}
The last inequality uses the Riemann sum bounds: For $s_l \ge 1$, $l \in [0, m]$, $\sum_{j=s_l}^{s_{l+1}-1} \alpha_j^2 \le \frac{1}{s_l}$ and $e^{-\sum_{i=s_l}^{k-1}\alpha_i} \le \frac{s_l + 1}{k+1}\le 2 \frac{s_l}{k}$.

Let $s_l = k^{\frac{2^l -1}{2^l}}$ for $l \in [0,m-1]$, we have $\frac{1}{\sqrt{s_l}} \frac{s_{l+1}}{k} = \frac{1}{\sqrt{k}}$. Then, setting $\frac{1}{\sqrt{s_{m}}} \le \frac{\sqrt{2}}{\sqrt{k}}$ gives $ k^{\frac{2^m -1}{2^m}} \ge \frac{k}{2}
\Leftrightarrow m \ge \log_2(\log_2 k)$ when $k>1$.
%


$\|\be_1\|_{L_{2q}} = \|\bz_0 - \bar{z}\|_{L_{2q}}$ gives $\|\be_1\|_{L_{2q}} \le 2 M_{2q}(\bz)$. Furthermore, when $k \ge 4$, we have $\log_2(\log_2 k) \ge 1$. Then,
we choose $m = \ceil{\log_2(\log_2 k)}$  to obtain 
\begin{align*}
\|\be_k \|_{L_{2q}} &\le (6 \sqrt{2} C_q + 4  M_{2q}(\bz)) \frac{1}{\sqrt{k}} \log_2(\log_2 k).
\end{align*}
%

Plugging the bounds from Proposition~\ref{prop:mixingMatrices} and the monotonicity of $L_q$ norm 
completes the proof.
\hfill\QED

\subsection{Proof of Proposition ~\ref{prop:bias}}

For $0\le i \le j$,
\begin{align} \label{eq:autocov}
\|R[j-i] \|_2
&= \left\|\bbE\left[(\by[j] - \bbE[\by[j] \vert \cF_i^+])\by[i]^\top \right]  \right. \nonumber \\
\nonumber
& \left. \quad \quad + \bbE[\bbE[\by[j] \vert \cF_i^+]\by[i]^\top] \right\|_{2} \nonumber \\
&\le \left \|\bbE\left[(\by[j] - \bbE[\by[j] \vert \cF_i^+])\by[i]^\top \right] \right \|_{2}  \nonumber \\
& \quad \quad + \left\| \bbE\left[\bbE[\by[j] \vert \cF_i^+]\by[i]^\top \right] \right\|_{2}. 
\end{align}

Then, use Jensen's inequality, norm inequalities and the Cauchy-Schwarz inequality to bound the first term of \eqref{eq:autocov}:
\begin{align*}
& \left\|\bbE\left[(\by[j] - \bbE[\by[j] \vert \cF_i^+])\by[i]^\top \right] \right\|_{2} \\
 &\le \bbE[\| (\by[j] - \bbE[\by[j] \vert \cF_i^+])\by[i]^\top\|_2] \\
 & \le  \bbE[ \|\by[j] - \bbE[\by[j] \vert \cF_i^+] \| \|\by[i]\|] \\
 & \le \sqrt{\bbE[ \|\by[j] - \bbE[\by[j] \vert \cF_i^+]\|^2 } \sqrt{\bbE[\|\by[i]\|^2]} \\
 &\le M_2(\by) \gamma_2(j-i, \by).
\end{align*}

For the second term of \eqref{eq:autocov}, we use the independence of $\bbE[\by[j] \vert \cF_i^+] $  and $\by[i]$, combining with the zero-mean assumption to obtain $
\bbE \left[\bbE[\by[j] \vert \cF_i^+]\by[i]^\top \right] = 0  $.

The bounds when $0\ge i \ge j$ are derived similarly. Overall, we have 
$
\|R[k]\|_2 \le  M_2(\by)\gamma_2(|k|,\by)
$
for all integers, $k$.

The rest now follows from bias calculations from \cite{lamperski2023nonasymptotic}.
\hfill\QED

\subsection{Proof of Theorem~\ref{thm:highProb}}

For all $\epsilon >0$ and all $q\ge 1$, Markov's inequality, followed by Theorem~\ref{theorem_main} implies:
\begin{align}
  \bbP(\|\bx_k-\bar{z}\|>\epsilon)
  & \le \epsilon^{-q}\|\bx_k-\bar{z}\|_{L_q}^q
  \label{eq:markovBound} \le \epsilon^{-q} f_k^q q^{rq}
\end{align}
where
$
f_k=c\frac{\log_2(\log_2(k))}{\sqrt{k}}.
$
The logarithm of the upper bound in (\ref{eq:markovBound}) is:
$
  g(q)=q\ln(\epsilon^{-1}f_k)+rq\ln q,
$
which is convex in $q$. The global minimum is given by $q^\star=e^{-1}\left(\frac{\epsilon}{f_k}\right)^{\frac{1}{r}}$
with $q^\star \ge 1$ when $\epsilon \ge f_k e^r$.
Thus, as long as $\epsilon\ge f_k e^{r}$, we have
$$
\bbP(\|\bx_k-\bar{z}\| >\epsilon)\le (\epsilon^{-1}f_k (q^\star)^r)^{q^\star}=e^{-\frac{r}{e}\left(\frac{\epsilon}{f_k}\right)^{\frac{1}{r}}}.
$$

Fix $\nu \in (0,1)$. Then $e^{-\frac{r}{e}\left(\frac{\epsilon}{f_k}\right)^{\frac{1}{r}}}\le \nu$ iff
$
\epsilon \ge f_k e^r \frac{\left(\ln \nu^{-1}\right)^r}{r^r}.
$
Thus, if $\epsilon= f_k e^r\max\left\{1, \frac{\left(\ln \nu^{-1}\right)^r}{r^r}\right\}$, $\bbP(\|\bx_k-\bar{z}\|>\epsilon)\le\nu$. 
\hfill\QED

\subsection{Proof of Proposition~\ref{prop:comparison}}
\label{ss:pf:comparison}
Let $\bzeta_i[k]=\bv$ be a $\sigma$-sub-Gaussian random variable. 
We follow the basic steps of Proposition 2.5.2 in \cite{vershynin2018high}, but track the constants more specifically.
\begin{align*}
  \bbE[|\bv|^q]
  &=\int_0^{\infty}(\bbP(\bv>\epsilon^{1/q})+\bbP(-\bv>\epsilon^{1/q}))d\epsilon.
\end{align*}
For all $\lambda >0$, we have:
\begin{multline*}
  \bbP(\bv >\epsilon^{1/q})
  =\bbP(e^{\lambda \bv} >e^{\lambda\epsilon^{1/q}}) 
  \le e^{-\lambda \epsilon^{1/q}}\bbE[e^{\lambda \bv}]
\\
  \le
e^{-\lambda \epsilon^{1/q}+\frac{1}{2}\lambda^2\sigma^2}
  \:\overset{\lambda=\epsilon^{1/q}/\sigma^2}{\implies} \:  \bbP(\bv >\epsilon^{1/q})  \le e^{-\frac{\epsilon^{2/q}}{2\sigma^2}}.
\end{multline*}
The same upper bound holds for $\bbP(-\bv>\epsilon^{1/q})$. Therefore,
\begin{align*}
  \bbE[|\bv|^q]&\le 2\int_0^{\infty} \exp\left(-\frac{\epsilon^{2/q}}{2\sigma^2}\right)d\epsilon
              =q(\sqrt{2}\sigma)^q \Gamma(q/2)\\
               &\le 3 q (\sqrt{2}\sigma)^q (q/2)^{q/2} 
               = e q\sigma^q q^{q/2}
\end{align*}
where the second inequality uses the Stirling bound $\Gamma(x)\le e x^x$ for $x\ge 1/2$.
It follows that
$
  \|\bv\|_{L_q} \le e^{\frac{1}{q}} q^{\frac{1}{q}} \sigma \sqrt{q}
  \le e^2\sigma \sqrt{q}\le 8\sigma \sqrt{q},
$
where the second inequality uses that $q^{\frac{1}{q}}\le e$ for $q\ge 1$.
Now, using the triangle inequality gives $M_q(\bzeta)\le 8m\sigma \sqrt{q}$. 

Using the bound on $M_q(\bzeta)$, we can bound $M_q(\by)$. Indeed, $\|h[k-\ell]\bzeta[\ell]\|_2\le \|h[k-\ell]\|_2 \|\bzeta[\ell]\|_2$ implies that
$
\|\by[k]\|_{L_q}\le M_q(\bzeta)\sum_{\ell=0}^{\infty}\|h[\ell]\|_2.
$
The bound on $M_q(\by)$ now follows.

To bound $\Gamma_{d,q}(\by)$, let $\cF_k=\sigma\{\bzeta[\ell]|\ell\le k\}$ and $\cF_k^+=\sigma\{\bzeta[\ell]|\ell > k\}$. Then for $0\le \ell\le k$,
\begin{align*}
  \left\|\by[k]-\bbE[\by[k]|\cF_{k-\ell}^+]\right\|_{L_q}&=\left\|\sum_{p = -\infty}^{k-\ell}h[k-p]\bzeta[p]\right\|_{L_q} \\
  &\le M_{q}(\bzeta)\sum_{p=\ell}^{\infty}\|h[p]\|_2.
\end{align*}
Summing the bound above over $l$ completes the proof.
\hfill\QED

\section{Simulation} \label{sec:sim}

To verify the obtained error bounds for the two classical spectrum estimators, samples are generated from measurements of a finite-state Markov chain.  
The stochastic process $\by[k] \in \{0,1\}$ corresponds to the sequence of states generated 
 via the Markov chain with transition matrix 
$
P = \begin{bmatrix}
0.3 & 0.7 \\
0.5 & 0.5
\end{bmatrix}.
$ In the simulation, we shift measurements by the mean $\bbE[\by[k]]=7/12$ to match the zero-mean assumption.

It can be proved that such an ergodic finite-state markov chain $\by[k]$ is L-mixing \cite{ gerencser2002new}. Proposition 4.1 of \cite{gerencser2002new} also calculates an upper bound of the $L$-mixing statistics $\Gamma_{d,q}(\by)$. Note that the Doeblin coefficient in our example is $\delta = \min\{0.3/(5/12), 0.5/(7/12)\} = 0.72$. Let $G_{\max} = \max_k\|\by[k]\|$. 
Therefore, $\Gamma_{d,4q}(\by) \le   4 G_{\max} \frac{1}{1- (1-\delta)^{\frac{1}{4q}}} \le \frac{4 G_{\max}}{\delta} 4q$.
 Furthermore, $M_{4q}(\by) \le G_{\max}$. This allows the explicit computation of the bound shown in Theorem~\ref{thm:highProb}.

We use Bartlett and Welch estimators to esimate power spectral density $\Phi(s)$ of the stochastic process at $s= \frac{1}{2}$. 
The convergence results are shown in Fig. \ref{fig:errors}, which present the errors of the algorithms ($\|\hat \bPhi_k(s) -\bbE[\hat \bPhi_k(s)]\|$) and the theoretical bound in Theorem \ref{thm:highProb}. We set $\nu = 0.1$, meaning the theoretical bound holds with probability $0.9$. We can see the empirical errors are all well below the theoretical bound. We are aware that these bounds are quite conservative. Getting a tighter bound is likely possible by improving some of the bounding techniques but not in the scope of this study.

  \begin{figure}
    \begin{minipage}[b]{\columnwidth}
      \centering
      \includegraphics[width=.7\textwidth]{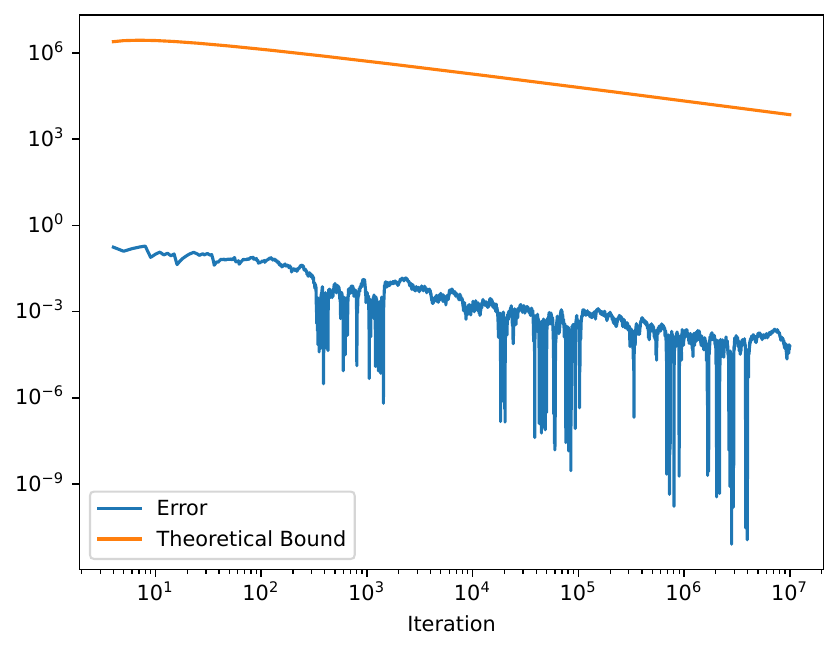}
      \subcaption{\textbf{Bartlett Estimator}. $M =5, L = 10^7$}
    \end{minipage}
    \begin{minipage}[b]{\columnwidth}
      \centering
      \includegraphics[width=.7\textwidth]{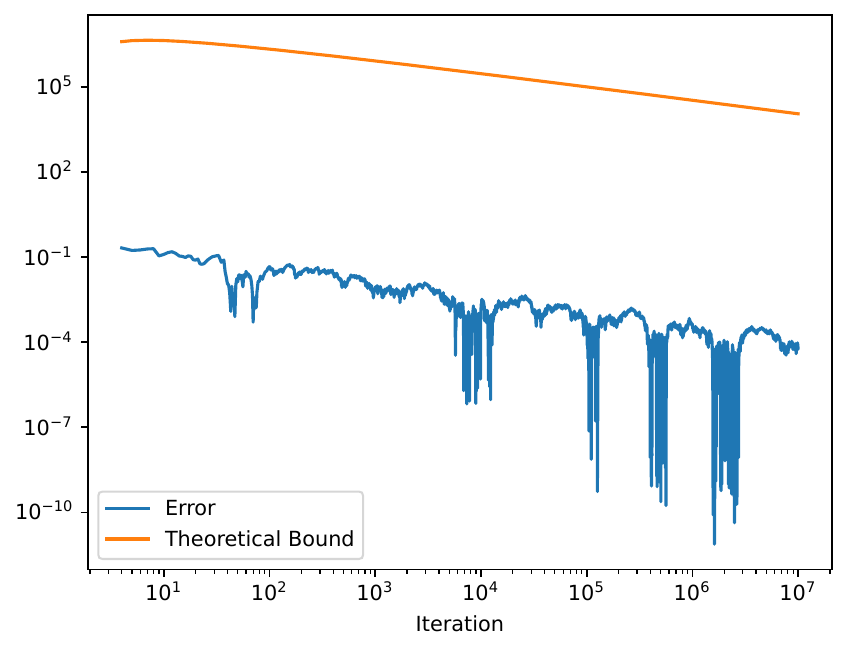}
       \subcaption{\textbf{Welch Estimator}. Hann Window, $M =16, K=8, L = 10^7$}
    \end{minipage}
	\caption{Concentration of estimate to its mean on finite Markov chain data}
	\label{fig:errors}
  \end{figure}

\section{Conclusion and Future Work}
\label{sec:conclusion}

In this work, we showed that the finite-time convergence rate for two classical spectrum estimators is of order $O(\frac{1}{\sqrt{k}} \log_2(\log_2 k))$, where $k$ is the number of chunks of data used in the algorithm. The error bounds corresponding to the variance of the estimators can be quantified by the $L$-mixing properties of the data. For Bartlett estimator, the concentration bound is independent of the window length $M$, while for Welch estimator, it depends on the ratio between $K$ and $M$, which is usually fixed as $0.5$ in practice. 

One limitation is that the zero-mean assumption on the time series often will not hold. 
 Therefore, in the future, we will extend the current analysis to more general time series which are not necessarily zero-mean. Another future direction is to improve the constant factors. 








\printbibliography

\end{document}